\documentclass[preprint,authoryear,review,10pt]{elsarticle2}

\usepackage{amssymb}
\usepackage[hyphens]{url}
\usepackage[breaklinks=true]{hyperref}
\usepackage[utf8]{inputenc}
\usepackage{array}
\usepackage{booktabs}
\usepackage{float}

\usepackage{ctable}
\usepackage{tabularx}
\usepackage{array}
\newcolumntype{P}[1]{>{\centering\arraybackslash}p{#1}}
\newcolumntype{+}{>{\global \let \currentrowstyle \relax}}
\newcolumntype{^}{>{\currentrowstyle }}

\usepackage[cmex10]{amsmath}

\hypersetup{linktocpage=true, urlcolor=magenta,linkcolor=blue,citecolor=blue, colorlinks=true,bookmarksopen=true,bookmarksopenlevel=0}
\newtheorem{thm}{Theorem}[section]                     
               
\newtheorem{lem}[thm]{Lemma}                                  
\newtheorem{prop}[thm]{Proposition}

\newenvironment{proof}{%
{\noindent \bf Proof. }%
}{%
\hfill$\Box$\\%
}

\def\N{\mathbb{N}}

\def \x {{\bf x}}

\usepackage{lscape}

\begin{document}

\begin{frontmatter}



\title{Influence of bovines and rodents in the spread of schistosomiasis across the ricefield-lakescape of Lake Mainit, Philippines: An Optimal Control Study
}

\author[csumath]{J. Arcede\corref{cor1}}
\author[lamfa,nuol]{B. Doungsavanh}
\author[iitbio]{L.A. Esta\~no}
\author[csubio]{J.C. Jumawan}
\author[csubio]{J.H. Jumawan}
\author[icj]{Y. Mammeri}

\address[csumath]{Caraga State University, Department of Mathematics,  Butuan City, Philippines}
\address[lamfa]{Universit\'e de Picardie Jules Verne, LAMFA CNRS UMR 7352, 80069 Amiens, France}
\address[nuol]{National University of Laos, Department of Mathematics and statistics, Vientiane, Laos}
\address[iitbio]{Department of Biological Sciences, Mindanao State University-Iligan Institute of Technology }
\address[csubio]{Caraga State University, Department of Biology,  Butuan City, Philippines}
\address[icj]{Universit\'e Jean Monnet, CNRS, Ecole Centrale de Lyon, INSA Lyon, Universite Claude Bernard Lyon 1, Institut Camille Jordan UMR5208, 42023 Saint-Etienne, France}

\cortext[cor1]{Corresponding author: jparcede@carsu.edu.ph}

\begin{abstract}
Schistosomiasis remains a persistent challenge in tropical freshwater ecosystems, necessitating the development of refined control strategies. Bovines, especially water buffaloes, are commonly used in traditional farming practices across rural areas of the Philippines. Bovines, however, are the biggest reservoir hosts for schistosome eggs, which contribute to the active transmission cycle of schistosomiasis in rice fields. We propose a mathematical model to analyze schistosomiasis dynamics in rice fields near the Lake Mainit in the Philippines, an area known for endemic transmission of schistosomiasis, focusing on human, bovine, and snail populations. Rodents, although considered, were not directly included in the control strategies. Grounded in field data, the model, built on a system of nonlinear ordinary differential equations, enabled us to derive the basic reproduction number and assess various intervention strategies. The simulation of optimal control scenarios, incorporating chemotherapy, mollusciciding, and mechanical methods, provides a comparative analysis of their efficacies. The results indicated that the integrated control strategies markedly reduced the prevalence of schistosomiasis. This study provides insights into optimal control strategies that are vital for policymakers to design effective, sustainable schistosomiasis control programs, underscored by the necessity to include bovine populations in treatment regimens.
\end{abstract}


\end{frontmatter}


\section{Introduction}

Lake Mainit is a large natural freshwater lake located in the northeastern part of the island of Mindanao in the Philippines. It is the fourth-largest lake in the country, covering an area of approximately 173.48 square kilometers.
The lake is situated in the provinces of Agusan del Norte and Surigao del Norte, surrounded by lush forests and mountain ranges. It is an important water source for irrigation, fisheries, and other livelihood activities for the local communities living around the lake. It is also home to diverse flora and fauna, including several endemic and endangered species. However, like many freshwater bodies, it is also threatened by pollution, deforestation, and other human activities that can negatively impact its ecological health and the communities that rely on it \citep{tumanda2003limnological,baluyut1983stocking}. 

One significant threat to Lake Mainit and its people is the incidence of schistosomiasis, which is endemic in the area since 1947 \citep{pesigan1947results}. The snail species responsible for transmitting the parasite to humans, called \textit{Oncomelania hupensis quadrasi}, has been found and the lake and its adjacent ricefields has been identified as a major transmission site for the disease, particularly in the municipalities of Jabonga, Mainit, Alegria, Tubay, and Sison in Agusan del Norte province \citep{abao2019schistosoma,estano2023prevailing,jumawan_prevalence_2021,jadap2015geospatial}. During the rainy season, the flooded interface between the lake and rice fields may promote zoonotic disease transmission via bovine fecal matter and snails hosting various parasitic species \citep{Jumawanzoonotic}. Surveys have linked snails and bovines in disease transmission within rice fields and other bovine-related parasitic diseases \citep{estano2023prevailing}. Farmers and lakeshore residents face occupational risks of schistosomiasis due to exposure in water bodies like irrigated canals, rice paddies, swamps, and residential areas, where snails and bovines thrive \citep{abao2019schistosoma, jumawan2021prevalence}. 
A preliminary study conducted on brown rats in the ricefields of five barangays near Lake Mainit reveals that the rodents carry hookworms, tapeworms, and ectoparasites, but there is no indication of \textit{Schistosoma} infection \citep{abao2021zoonotic}. Despite the study's limited number of sites and sampling constraints, further extensive data collection is necessary to confirm these initial results. However, if the opposite were true, it could exacerbate the situation significantly.  Notably, there were no reported surveys for schistosomiasis-infected goats, dogs, or swine nearby.

According to WHO, schistosomiasis is the most prevalent water-borne parasitic disease affecting millions worldwide, particularly in areas with poor sanitation and limited access to safe water sources \citep{who2013schistosomiasis}. The epidemiological statistics associated with schistosomiasis are impressive: 800 million people are at risk in 78 countries, mostly concentrated in sub-Saharan Africa; 230 million are infected, and the disease is responsible for between 1.7 and 4.5 million disability-adjusted life years \citep{boissier2019schistosoma}.

Blood flukes (trematode worms) of the genus \textit{Schistosoma} cause schistosomiasis \citep{ross2017new}.  Schistosomes have a complicated life cycle. However, in this paper, we follow the following life cycle of \textit{Schistosoma} as revealed by \cite{nelwan2019schistosomiasis} and \cite{zhou2016multi}.
The parasitic worm involves two hosts: a definitive host (humans, carabao, other mammals) and an intermediate host snail (\textit{Oncomelania hupensis quadrasi}). The life cycle can be divided into several stages:

\begin{enumerate}
    \item Infection of mammalian host and sexual reproduction: Cercariae, the free-swimming larval forms of \textit{Schistosoma} enter the mammalian skin and shed their forked tail, forming a schistosomula.  Sexual reproduction occurs in the definitive host --mammals (humans, bovines, rodents, dogs). The schistosomula migrate throughout the mammal’s tissues through blood circulation. Schistosomula grow into mature schistosomes, mate, and produce eggs.
    \item Release of eggs into the environment: The eggs are then passed out of the mammalian host’s body and into the environment through feces or urine. In freshwater, these eggs form miracidia, which hatch and infect intermediate host snails. 
    \item Asexual reproduction: Occurs in the intermediate host: \textit{Schistosoma japonicum} infects snails of the genus \textit{Oncomelania},  \textit{S. mekongi} infects snails of the genus \textit{Neutricula};  \textit{S. mansoni} infects snails of the genus \textit{Biomphalaria}. After infiltration of the snail host, the miracidium removes the ciliated plates, develops into a mother sporocyst, and then produces daughter sporocysts. Daughter sporocysts produce either cercaria (cercariogenous sporocysts) or more daughter sporocysts (sporocystogenous sporocysts). The cercaria that emerges from snails looks for a suitable mammalian host to penetrate, starting the cycle over again.
\end{enumerate}

In many parts of the country, rice farming remains largely unmechanized, relying instead on water buffaloes (carabaos), which are recognized as major reservoir hosts for schistosomiasis \citep{olveda2019schistosomiasis}. Rice farming in areas endemic for schistosomiasis poses occupational risk as snail intermediate hosts thrive in this nidus, effectively perpetuating the life cycle of \textit{Schistosoma} \citep{jumawan2021prevalence}. Mass drug administration (MDA) of praziquantel (PZQ) in humans is the primary method of treating schistosomiasis in the country \citep{olveda2019schistosomiasis}. However, PZQ only targets mature worms and does not provide immunity \citep{WHO2020schistosomiasis}.

Recent reports have highlighted the effectiveness of combining interventions such as human mass chemotherapy, mollusciciding, and bovine vaccination in endemic villages \citep{ross2023first}. Such comprehensive measures are critical in ecosystems characterized by double-cropping rice irrigation schemes, correlated with a higher prevalence of \textit{Schistosoma haematobium }\citep{coulibaly2004comparison}. Notably, in Japan, snail habitat modification through the construction of concrete irrigation channels has proven effective in disrupting the parasite's lifecycle by preventing snail proliferation \citep{bay2022total}. Furthermore, mechanized farming has emerged as a promising intervention, effectively reducing the population of the largest reservoir host for the disease \citep{gordon2019asian}  in many Asian countries. However, while the success of these interventions in typical agricultural settings is documented, strategies tailored to unique lake-ricefield environments, such as Lake Mainit, are not commonly reported in the literature

A systematic review by \cite{LOWE2023106873} highlighted several studies focusing on host heterogeneity and the interplay between humans, snails, and bovines, most of which are concentrated on scenarios in China. Another one by \cite{chiwendu2022} incorporates heterogeneous host transmission pathways, particularly the roles of bovines, rodents, and environmental factors. However, the paper emphasizes reducing the snail population as a control strategy, which might not be feasible in all endemic areas due to ecological and economic constraints. Furthermore, \cite{Lutz2023} focuses on a comprehensive mathematical model that evaluates integrated control strategies considering various environmental factors, such as temperature and chemical degradation, without specific emphasis on the role of bovines and rodents. All the papers mentioned are not set in a ricefield-lakescape environment, and most do not specify a particular geographic focus. Therefore, a field study is needed to validate the model's predictions and refine the parameters used, ensuring the model's applicability to real-world scenarios.

Our primary objective in this paper is to elucidate the dynamics of schistosomiasis within ricefields adjacent to Lake Mainit by applying a system of ordinary differential equations underpinned by field data, we not only derive the basic reproduction number, $\mathcal{R}_0$, but also propose optimized control interventions grounded in control theory principles. 
To achieve this goal, we employed mathematical modeling and a conceptual approach to comprehensively analyze schistosomiasis transmission in the Lake Mainit region of the Philippines. The design of our model was informed by on-site data collected during parasitological surveys and snail diversity research conducted between 2020 and 2023 \citep{estano2023prevailing, jumawan_prevalence_2021,  Jumawanzoonotic}. Following parameter estimation, we derived the basic reproduction number, $\mathcal{R}_0$, while considering the uncertainty associated with the model's parameters.

Given that \textit{S. japonicum} is a zoonotic parasite, its categorization as a spillover pathogen, multi-host pathogen, or true multi-host pathogen varies depending on the location and scale of the observed multi-host parasite community, such as within ricefields. We elucidate the most significant host species in each region by applying our proposed conceptual framework, identifying maintenance and essential hosts for transmission as defined in the Materials and Methods section. Additionally, we investigate the potential effects of interventions targeting various zoonotic reservoirs to control and eradicate \textit{S. japonicum}  and \textit{O. quadrasi} in the lake. Our research illustrates how mathematical models and conceptual frameworks can be applied to empirical data to characterize systems and uncertainties, and how they are pertinent for controlling schistosomiasis in Lake Mainit, Philippines.

The paper is structured as follows: Section 2 outlines our methodology, explaining the model and its qualitative analysis. This section also provides the closed-form equation for the reproductive number and discusses optimal strategies for reducing transmission. Section 3 presents numerical experiments comparing optimal strategies in scenarios with and without the presence of uncontrolled rodents. Finally, Section 4 briefly discusses measures to limit the outbreak, and we conclude with Section 5.

\section{Material and methods}

\subsection{Description of the mathematical model}

Our model takes into account the life cycle of schistosomiasis that involves complex stages as described by \cite{gordon2019asian}.
The life cycle begins when eggs are passed in the feces or urine of infected mammals who have come into contact with contaminated water bodies. The eggs hatch in water, releasing larvae called miracidia.
Once released into freshwater, miracidia actively seek out specific freshwater snails, which serve as intermediate hosts. They penetrate the tissues of these snails and develop into sporocysts.
Sporocysts mature and produce cercariae, which are released from the snails into the water.
Cercariae are free-swimming larvae that actively penetrate the intact skin of mammals who come into contact with contaminated water. Once inside the human body, cercariae transform into schistosomulae.
Schistosomulae migrate and mature into adult worms. 
Male and female adult worms pair up and produce eggs. The eggs traverse the tissues of the intestines or urinary tract and are released into the environment through feces or urine, perpetuating the cycle.

\begin{figure}[htbp]
\centering
\includegraphics[scale=0.4]{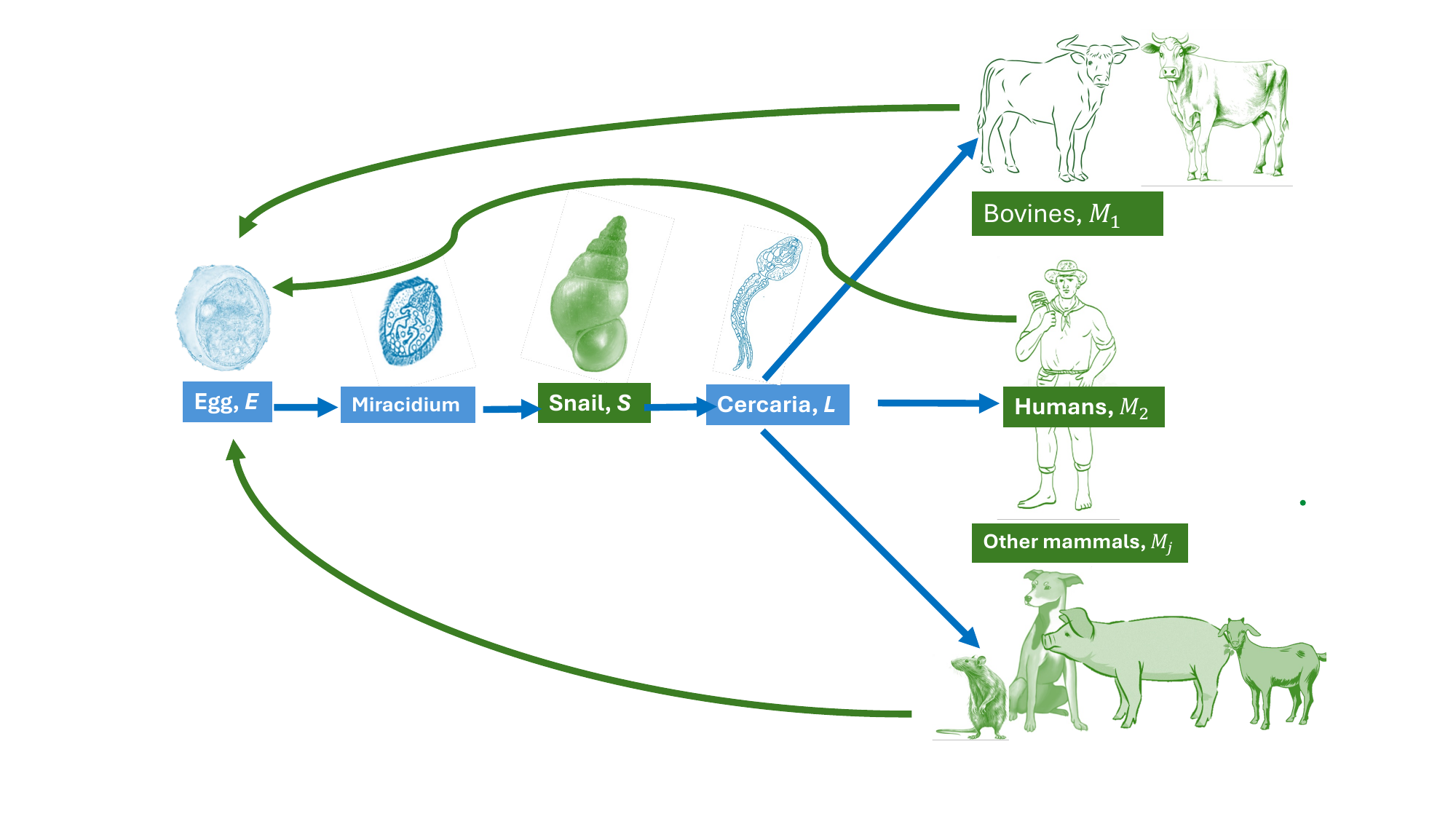} 
\\[-0.5cm]\caption{Compartmental representation of the dynamics of Schistosomiasis life cycle in snails and mammals.
The schistosomiasis life cycle starts with infected mammals shedding eggs in water. Miracidia larvae hatch from these eggs and target freshwater snails. 
Within snails, miracidia produce cercariae. Free swimming-cercariae penetrate the skin of mammals in water.}
\label{figflow}
\end{figure}

\newpage

We consider a compartmental model made with the above states: let $E$ be the number eggs of schistosomiasis, 
$L$ be the number larvae of schistosomiasis or cercariae,  $S_s$ and $S_i$ be the number susceptible and infected snails, 
for $1\leq j \leq m$, $M_{j,s}$ and $M_{j,i}$ be the number susceptible and infected mammals. 

\begin{equation}
  \label{e1}
  \left\{
      \begin{aligned}
        E'(t) & = \sum_{j=1}^{m} \theta_j M_{j,i}(t) - \omega_s S_s(t) E(t) - \mu_e E(t)
        \\
        S_s'(t) & =  -\omega_s \beta_s S_s(t) E(t) + \alpha_s (S_s(t) + S_i(t)) \left( 1 - \frac{S_s(t) + S_i(t)}{\kappa_s}\right) - \mu_s S_s(t)
        \\
        S_i'(t) & =  \omega_s \beta_s S_s(t) E(t) - (\mu_s + \gamma_s) S_i(t)
        \\
        L'(t) &= \nu_s S_i(t) - \sum_{j=1}^{m} \omega_j M_{j,s}(t) L(t) - \mu_l L(t)
        \\
        M_{j,s}'(t) & =  -\omega_j \beta_j L(t) M_{j,s}(t) + \alpha_j - \mu_j M_{j,s}(t), \hskip 1cm 1 \leq j \leq m
        \\
        M_{j,i}'(t) & =  \omega_j \beta_j L(t) M_{j,s}(t) - (\mu_j + \gamma_j) M_{j,i}(t), \hskip 1.1cm 1 \leq j \leq m   
      \end{aligned}
    \right.
\end{equation}
Note that the total population of snails is $S= S_s+S_i$ and the total population of mammals is $M_j = M_{j,s}+M_{j,i}$.
Parameters are described in the table \ref{tabparam}.

\subsection{Qualitative study}
It is standard to check that there exists a unique global in time solution
where the population are nonnegative and bounded.
The system of ordinary differential equations is rewritten as
$ U^\prime(t)= F(U(t), t)$ with $$U(t)= (E, S_s, S_i, L, M_{1,s}, M_{1,i}, \cdots, M_{m,s}, M_{m,i})^T$$ and 
\begin{align*}
    F(U(t), t)= 
    \begin{pmatrix}
    \sum_{j=1}^{m} \theta_j M_{j,i} - \omega_s S_s E - \mu_e E\\
     -\omega_s \beta_s S_s E + \alpha_s (S_s + S_i) \left( 1 - \frac{S_s + S_i}{\kappa_s}\right)- \mu_s S_s\\
     \omega_s \beta_s S_s E - (\mu_s + \gamma_s) S_i\\
     \nu_s S_i - \sum_{j=1}^{m} \omega_j M_{j,s} L - \mu_l L\\
     -\omega_1 \beta_1 L M_{1,s} + \alpha_1 - \mu_1 M_{1,s}\\
     \omega_1 \beta_1 L M_{1,s} - (\mu_1 + \gamma_1) M_{1,i}\\
     \vdots\\
      -\omega_m \beta_m L M_{m,s} + \alpha_1 - \mu_j M_{m,s}\\
      \omega_m \beta_m L M_{m,s} - (\mu_m + \gamma_) M_{m,i}
    \end{pmatrix}
\end{align*} 
\begin{thm}\label{thm1}
The domain $\Omega$ defined by 
\begin{align*}
    \Omega :=  \Bigg\{ & (E, S_s, S_i, L, M_{1,s}, M_{1,i}, \cdots, M_{m,s}, M_{m,i}) \in \mathbb{R}^{2m+4}; \\ & 0\le E\le \max \left(\dfrac{\sum_{j=1}^{m} \theta_j \max \left(\dfrac{\alpha_i}{\mu_j}, M_0\right)}{\mu_e}, E_0 \right), \,\, 0 \le L\le\max \left(\dfrac{\nu_s \max(\kappa_s (\alpha_s-\mu_s),S_0)}{\mu_l}, L_0\right), \\
    &0 \le S_s+S_i \le \max\left(\kappa_s (\alpha_s-\mu_s), S_0\right), \,\, 0 \le M_{j,s}+M_{j,i} \le \max \left(\dfrac{\alpha_j}{\mu_j}, M_0\right),
    \,\, 1 \le j \le m \Bigg\}
\end{align*}
is positively invariant. In particular, for any initial datum in $\Omega$, there exists a unique global in time solution in $\mathcal{C} \left(\mathbb{R}_{+}, \Omega\right)$.
\end{thm}
\begin{proof}
Considering the initial value problem 
$$ U^\prime(t)= F(U(t), t) \quad \text{where}  \quad U(0)=U_0,$$
where the right-hand-side $F$ satisfies the local Lipschitz condition, then the Cauchy-Lipschitz theorem ensures the local well-posedness. 
\\
On the other hand, we note that for all $(E, S_s, S_i, L, M_{1,s}, M_{1,i}, \cdots, M_{m,s}, M_{m,i}) \in \Omega$, we have 
\begin{eqnarray*}
f_1(0, S_s, S_i, L, M_{1,s}, M_{1,i}, \cdots, M_{m,s}, M_{m,i}) &=& \sum_{j=1}^m \theta_j M_{j,i}\ge 0  \\
f_2(E, 0, S_i, L, M_{1,s}, M_{1,i}, \cdots, M_{m,s}, M_{m,i}) &=& \alpha_s S_i \left(1-\dfrac{S_i}{\kappa_s}\right) \ge 0\\
f_3(E, S_s, 0, L, M_{1,s}, M_{1,i}, \cdots, M_{m,s}, M_{m,i}) &=& \omega_s \beta_s S_s E\ge 0\\
f_4(E, S_s, S_i, 0, M_{1,s}, M_{1,i}, \cdots, M_{m,s}, M_{m,i}) &=& \nu_s S_i \ge 0
\\
f_{5,s}(E, S_s, S_i, L, 0, M_{1,i}, \cdots, M_{m,s}, M_{m,i}) &=& \alpha_1 \ge 0
\\
f_{5,i}(E, S_s, S_i, L, M_{1,s}, 0, \cdots, M_{m,s}, M_{m,i}) &=& \omega_1 \beta_1 L M_{1,s} \ge 0\\
\vdots\\
f_{m+4,s}(E, S_s, S_i, L,M_{1,s}, M_{1,i}, \cdots, 0, M_{m,i}) &=& \alpha_m \ge 0\\
f_{m+4,i}(E, S_s, S_i, L, M_{1,s}, M_{1,i}, \cdots, M_{m,s}, 0) &=& \omega_m \beta_m L M_{m,s}\ge 0.
\end{eqnarray*}
We deduce that, starting from a nonnegative initial data, 
the corresponding solution remains nonnegative.
\\
Let $S = S_s + S_i$ and $ M_j  = M_{j,s} + M_{j,i}$ be the total population of snails and mammals respectively. We have
\begin{eqnarray*}
   S^\prime(t) = S_s'(t) + S_i'(t) &=&\alpha_s (S_s(t) + S_i(t)) \left( 1 - \frac{S_s(t) + S_i(t)}{\kappa_s}\right) - \mu_s S_s(t)- (\mu_s + \gamma_s) S_i(t)\\
    &=&\alpha_s S(t) \left( 1 - \frac{S(t)}{\kappa_s}\right) - \mu_s S(t) -  \gamma_s S_i(t).
\end{eqnarray*}
Since the solution is nonnegative,
\begin{eqnarray*}
    S^\prime(t) & \le &\alpha_s S(t) \left( 1 - \frac{S(t)}{\kappa_s}\right) - \mu_s S(t) 
    \\
    & \le & \left( (\alpha_s-\mu_s) - \dfrac{\alpha_s}{\kappa_s}  S(t)\right) S(t),
\end{eqnarray*}
and the Grönwall lemma implies that  $S(t) \le \dfrac{\kappa_s(\alpha_s-\mu_s ) S_0 e^{(\alpha_s-\mu_s)t}}{\kappa_s(\alpha_s-\mu_s ) + \alpha_s S_0 e^{(\alpha_s-\mu_s)t}} \le \max(\kappa_s (\alpha_s-\mu_s),S_0)$.
\\
Concerning the mammals, we write
\begin{eqnarray*}
   M'_j(t) = M_{j,s}'(t) +(t) M_{j,i}'(t) &= &  \alpha_j-\mu_j M_{j,s}(t) -(\mu_j+\gamma_j)M_{j,i}(t)\\
   & =& \alpha_j-\mu_j M_j(t) - \gamma_j M_{j,i}(t)
 \end{eqnarray*}  
 and by positivity of the solution
 \begin{eqnarray*}
   M_j'(t) &\le & \alpha_j - \mu_j M_j(t).
\end{eqnarray*}
From Grönwall's lemma, $M(t)  \le \dfrac{\alpha_j}{\mu_j} + e^{-\mu_jt} \left( M_j(0) - \dfrac{\alpha_j}{\mu_j} \right) \le \max \left(\dfrac{\alpha_j}{\mu_j}, M_0\right)$.\\
Because $S$ and $M_j$ are nonnegative and bounded,
\begin{eqnarray*}
    E'(t) & =  &\sum_{j=1}^{m} \theta_j M_{j,i} - \omega_s S_s E - \mu_e E\\
     & \le  & \sum_{j=1}^{m} \theta_j \max \left(\dfrac{\alpha_j}{\mu_j}, M_0\right) - \mu_e E
\end{eqnarray*}
and
\begin{eqnarray*}
     L'(t) &=& \nu_s S_i - \sum_{j=1}^{m} \omega_j M_{j,s} L - \mu_l L\\
     &\le & \nu_s \max(\kappa_s (\alpha_s-\mu_s),S_0) - \mu_l L.
\end{eqnarray*}
Finally, applying Grönwall's lemma again provides
\begin{eqnarray*}
    E(t)  & \le  & \dfrac{\sum_{j=1}^{m} \theta_j \max \left(\dfrac{\alpha_j}{\mu_j}, M_0\right)}{\mu_e} + e^{-\mu_e t} \left( E_0 - \dfrac{\sum_{j=1}^{m} \theta_j \max \left(\dfrac{\alpha_j}{\mu_j}, M_0\right)}{\mu_e} \right) 
  \\  &\le & \max \left(\dfrac{\sum_{j=1}^{m} \theta_j \max \left(\dfrac{\alpha_j}{\mu_j}, M_0\right)}{\mu_e}, E_0\right),
    \\
    L(t) & \le  & \dfrac{\nu_s \max(\kappa_s (\alpha_s-\mu_s),S_0)}{\mu_l} + e^{-\mu_l t} \left( L_0 - \dfrac{\nu_s \max(\kappa_s (\alpha_s-\mu_s),S_0)}{\mu_j} \right) \\  &\le &  \max \left(\dfrac{\nu_s \max(\kappa_s (\alpha_s-\mu_s),S_0)}{\mu_l}, L_0\right).
\end{eqnarray*}
\end{proof}

Formally, three equilibriums may exist:
\begin{itemize}
    \item The first disease free equilibrium
$DFE_1= (0,0,0,0,M_1^*,0, \cdots, M_m^*, 0).$
\item The second disease free equilibrium
$DFE_2=(0,S^*,0,0,M_1^*,0, \cdots, M_m^*, 0) $ 
   with $S^*=\dfrac{\kappa_s(\alpha_s-\mu_s)}{\alpha_s} , M_j^*=\dfrac{\alpha_j}{\mu_j}$ for $1 \le j\le m.$ 
   \item The endemic equilibrium $EE=(E^*, S_s^*, S_i^*, L^*, M_{1,s}^*, M_{1,i}^*, \cdots, M_{m,s}^*, M_{m,i}^*)$ with, for $1 \le j\le m$, 
   \begin{align*}
&E^*=\dfrac{\sum_{j=1}^m\dfrac{\omega_j \beta_j \theta_j \alpha_j L^*}{(\omega_j \beta_j L^*+\mu_j)(\mu_j+\gamma_j)}}{(\omega_s S_s^*+\mu_e)}\\
  &S_s^* = \dfrac{\sum_{j=1}^{m} \left( \dfrac{\omega_j \alpha_j}{\omega_j \beta_j L^* +\mu_j }+\mu_l\right)(\mu_s +\gamma_s )\mu_e}{\omega_s \nu_s \beta_s \sum_{j=1}^{m}\dfrac{\omega_j \beta_j \theta_j \alpha_j}{(\omega_j \beta_j L^* +\mu_j)(\mu_j+\gamma_j)}-\omega_s(\mu_s +\gamma_s ) \sum_{j=1}^{m} \left( \dfrac{\omega_j \alpha_j}{\omega_j \beta_j L^* +\mu_j}+\mu_l\right)} 
  \\
&S_i^* = \dfrac{\omega_s \beta_s S_s^* \sum_{j=1}^m\dfrac{\omega_j \beta_j \theta_j \alpha_j L^*}{(\omega_j \beta_j L^*+\mu_j)(\mu_j+\gamma_j)}}{(\omega_s S_s^*+\mu_e)(\mu_s+\gamma_s)}\\
&M_{j,s}^* = \dfrac{\alpha_j}{\omega_j \beta_j L^*+\mu_j}\\
&M_{j,i}^* = \dfrac{\omega_j \beta_j \alpha_j L^*}{(\omega_j\beta_j L^*+\mu_j)(\mu_j+\gamma_j)}.
\end{align*}
\end{itemize}
The basic reproduction number $\mathcal{R}_0$ can be computed thanks to the next generation matrix of the model.
Since the infected individuals are in $E, S_i, L, M_{j,i}$, 
the rate of appearance of new infections in each compartment $\mathcal{F}$  and  the rate of other transitions between all compartments $\mathcal{V}$ can be rewritten  as 
$$\mathcal{F} =\begin{pmatrix} 
  \sum_{j=1}^{m} \theta_j M_{j,i} -\omega_s S_s E \\
  \omega_s \beta_s S_s E \\
  \nu_s S_i  - \sum_{j=1}^{m} \omega_j M_{j,s} L \\
  \omega_1 \beta_1 L M_{1,s} \\
  \vdots \\
  \omega_m \beta_m L M_{m,s} 
 \end{pmatrix}, 
\; \mathcal{V} = \begin{pmatrix}
  \mu_e E \\ 
  (\mu_s + \gamma_s) S_i \\
  \mu_l L \\
  (\mu_1 + \gamma_1) M_{1,i}\\
  \vdots  \\
  (\mu_m + \gamma_m) M_{1,m} 
\end{pmatrix}.
$$
Thus,
$$
F = \begin{pmatrix}
  -\omega_s S_s & 0 & 0 & \theta_1 & \hdots & \theta_m \\ 
  \omega_s \beta_s S_s & 0 & 0 & 0 & \hdots & 0 \\ 
  0 & \nu_s & -\sum_{j=1}^{m} \omega_j M_{j,s} & 0 & \hdots & 0 \\
  0 & 0 & \omega_1 \beta_1 M_{1,s} & 0 & \hdots & 0 \\ 
  0 & 0 & \vdots & 0 & \hdots & 0 \\
  0 & 0 & \omega_m \beta_m M_{m,s} & 0 & \hdots & 0
 \end{pmatrix},$$
 $$
V= \begin{pmatrix}  
  \mu_e & 0 & 0 & 0 & \hdots & 0 \\ 
  0 & (\mu_s + \gamma_s) & 0 & 0 & \hdots & 0 \\
  0 & 0 & \mu_l & 0 & \hdots & 0\\
  0 & 0 & 0 & (\mu_1 + \gamma_1)  & & 0\\
  \vdots  &   & \ddots  &  & \ddots  & \vdots \\
  0  & 0  & 0  & 0 & & (\mu_m + \gamma_m) 
  \end{pmatrix}.
  $$
Therefore,  the next generation matrix is
$$ FV^{-1}=\begin{pmatrix}- \dfrac{ \omega_{s}S^*}{\mu_{e}} & 0 & 0 & \dfrac{\theta_{1}}{\gamma_{1} + \mu_{1}} & \cdots & \dfrac{\theta_{m}}{\gamma_{m} + \mu_{m}}\\
\dfrac{ \omega_{s}\beta_{s} S^*}{\mu_{e}} & 0 & 0 & 0 & \cdots& 0\\
0 & \dfrac{\nu_{s}}{\gamma_{s} + \mu_{s}} & \dfrac{- \sum_{j=1}^m  \omega_{j}M_{j}^*}{\mu_{l}} & 0& \cdots& 0\\
0 & 0 & \dfrac{\omega_{1} \beta_{1} M_{1}^*}{\mu_{l}} & 0 & \cdots &0\\
0& 0 & \vdots & 0 & \cdots  &0\\
0 & 0 & \dfrac{ \omega_{m}\beta_{m} M_{m}^*}{\mu_{l}} & 0 &\cdots& 0\end{pmatrix}
$$
and the characteristic equation of $FV^{-1}$ is
$$\lambda^{m-1}\left( \left(-\lambda-\dfrac{\sum_{j=1}^m \omega_j M_j^*}{\mu_l}\right) \left(-\lambda-\dfrac{\omega_s S^*}{\mu_e}\right)\lambda^{2}- \left(\dfrac{\omega_s \beta_s S^*}{\mu_s+\gamma_s}\right)\left(\dfrac{\nu_s}{\mu_e}\right)\sum_{j=1}^m \left(\dfrac{\omega_j \beta_j M_j^*}{(\mu_j+\gamma_j)}\right)\left(\dfrac{\theta_j}{\mu_l}\right)\right)=0.$$
We deduce that the basic reproduction number is
$$
\mathcal{R}_0 :=  
\sum_{j=1}^m \left(\frac{\omega_j \beta_j M_j^*}{\mu_j+\gamma_j} \right)\left( \frac{\theta_j}{\mu_e} \right)  \left(\frac{\omega_s \beta_s S^*}{\mu_s+\gamma_s} \right)  \left( \frac{\nu_s}{\mu_l} \right) , $$
with $M_j^* = \dfrac{\alpha_j}{\mu_j}, S^* = \dfrac{\kappa_s(\alpha_s-\mu_s)}{\alpha_s}.$

\begin{prop}
  \begin{enumerate}
     \item If $\alpha_s < \mu_s$, the disease-free equilibrium $DFE_1$ is locally asymptotically stable.
    \item If $\alpha_s > \mu_s$ and $\mathcal{R}_0 \leq 1$, the disease-free equilibrium $DFE_2$ is locally asymptotically stable.
    \item If $\alpha_s > \mu_s$ and  $\mathcal{R}_0 > 1$, the endemic equilibrium EE is locally asymptotically stable.
  \end{enumerate}
\end{prop}

The basic reproduction number $\mathcal{R}_0$ has a
biologically meaning when $\alpha_s > \mu_s$. It means that  $\left( \dfrac{\nu_s}{\mu_e} \right) \left(\dfrac{\omega_s \beta_s S^*}{\mu_s+\gamma_s} \right) $
represents the transmission rate due to the snails infected by eggs during the infection period of snails $\dfrac{1}{\mu_s+\gamma_s}$.
The term $\left( \dfrac{\theta_j}{\mu_l} \right) \left(\dfrac{\omega_j \beta_j M_j^*}{\mu_j+\gamma_j} \right)$
represents the transmission rate due to the mammals $j$ infected by larvae during the infection period of mammals $\dfrac{1}{\mu_j+\gamma_j}$.

\begin{thm}
  \begin{enumerate}
     \item If $\alpha_s < \mu_s$, the disease-free equilibrium $DFE_1$ is globally asymptotically stable.
    \item If $\alpha_s > \mu_s$ and $\mathcal{R}_0 \leq 1$, the disease-free equilibrium $DFE_2$ is globally asymptotically stable.
  \end{enumerate}
\end{thm}

\begin{proof} We use Theorem 9.2 of \cite{brauer2011} after rewritting the system as 
\begin{eqnarray*}
  x' &=&  -A x - f(x,y) \\ 
  y' &=& g(x,y)
\end{eqnarray*}
with $x = (E, S_i, L, M_{1,i}, M_{2,i}, \dots, M_{m,i})$, $y = (S_s, M_{1,s}, M_{2,s},\dots, M_{m,s})$, 
and
$$
  g(x,y) =  
  \begin{pmatrix}
   -\omega_s \beta_s S_s E + \alpha_s (S_s + S_i) \left( 1 - \frac{S_s + S_i}{\kappa_s}\right)- \mu_s S_s\\
   -\omega_1 \beta_1 L M_{1,s} + \alpha_1 - \mu_1 M_{1,s}\\
   \vdots\\ 
    -\omega_m \beta_m L M_{m,s} + \alpha_1 - \mu_j M_{m,s}
  \end{pmatrix},
$$ 
and
$$
  -A  =  
  \begin{pmatrix}
  - \omega_s S^* - \mu_e  & 0 & 0 & \theta_1 & \dots & \dots & \theta_m \\
   \omega_s \beta_s S_s & - (\mu_s + \gamma_s) & 0 & 0 & \dots & \dots & 0\\
   0 & \nu_s & - \sum_{j=1}^{m} \omega_j M_{j}^* - \mu_l & 0 & \dots & \dots & 0\\
   0 & 0 & \omega_1 \beta_1 M_{1}^* & - (\mu_1 + \gamma_1)  & 0 & \dots & 0 \\
   \vdots & \vdots & \vdots  & 0 & \ddots & \ddots & \vdots \\
   \vdots & \vdots & \vdots  & \vdots & \ddots & \ddots & 0 \\
   0 & 0 & \omega_m \beta_m M_{m}^* & 0 & \dots & 0 & - (\mu_m + \gamma_m)
  \end{pmatrix},
$$
such that $y^* =(0,M_1^*,\cdots, M_m^*)$ is globally asymptotically stable for the differential equation
$y'(t) = g(0,y(t))$ if $\alpha_s < \mu_s$ and , and on the other hand $y^* =(S^*,M_1^*,\cdots, M_m^*)$ is globally asymptotically stable for the differential equation
$y'(t) = g(0,y(t))$ if $\alpha_s > \mu_s$ and $\mathcal{R}_0 \leq 1$.
\end{proof}

\subsection{Optimal control strategy}
This paper aims to reduce the number of humans and bovines infected by controlling the population of the vector—\textit{Oncomelania quadrasi} snails, which harbor \textit{Schistosoma japonicum}. To achieve this goal, we consider three control inputs: $u_a$ represents the rate at which eggs, miracidia and cercariae are exposed to aquatic controls, such as installing sanitation facilities like toilets and using machinery like tractors; $u_s$ denotes the rate at which snails are exposed to molluscicides, while $u_\kappa$ signifies the rate at which the snail population (both susceptible and infected) is subjected to habitat modification. Finally, $u_j$ represents the rates at which both humans and bovines are exposed to chemotherapy treatment. Assume that all control inputs are piecewise continuous functions that take their value in a positively bounded set
$W = [0, u_{Amax}] \times [0, u_{Smax}] \times  [0, u_{Mmax,1} ] \times \cdots \times  [0, u_{Mmax,m_c}]$.\\
We consider the objective function
$$
J(u) = \int_0^T  \sum_{j=1}^{m_c} M_{j,i}(t) + A_a u_a^2(t) + A_s u_s^2(t) + A_{\kappa} u_{\kappa}^2(t) +  \sum_{j=1}^{m_c} A_j u_j^2(t) dt
$$
subject to, for $t \in [0,T]$,
\begin{equation}
  \label{e2}
  \left\{
      \begin{aligned}
        E'(t) & =  \sum_{j=1}^{m} \theta_j M_{j,i} - \omega_s S_s E - \mu_e E - u_a(t) E 
        \\
        S_s'(t) & =  -\omega_s \beta_s S_s E + \alpha_s (S_s + S_i) \left( 1 - \frac{(S_s + S_i)}{\kappa_s/u_{\kappa}(t)}\right) - \mu_s S_s - u_s(t) S_s
        \\
        S_i'(t) & =  \omega_s \beta_s S_s E - (\mu_s + \gamma_s) S_i - u_s(t) S_i
        \\
        L'(t) &= \nu_s S_i - \sum_{j=1}^{m} \omega_j M_{j,s} L - \mu_l L - u_a(t) L
        \\
        M_{j,s}'(t) & =  -\omega_j \beta_j L M_{j,s} + \alpha_j - \mu_j M_{j,s}, \hskip 2.8cm 1 \leq j \leq m
        \\
        M_{j,i}'(t) & =  \omega_j \beta_j L M_{j,s} - (\mu_j + \gamma_j) M_{j,i} - u_j(t) M_{j,i}, \hskip 1.1cm 1 \leq j \leq m  
      \end{aligned}
    \right.
\end{equation}
The variables $A_i$ are the positive weights associated with the control variables $u_i, i= a, s,\kappa, j$, respectively.
They correspond to the efforts rendered in exposing the eggs, cercariae, snails, human and bovines to our control strategy.
\begin{lem}
There exists an optimal control $u^*= \left(u^*_a(t), u^*_s(t), u^*_\kappa(t), u^*_1(t), \dots, u^*_{m_c}(t)\right)$ such that 
$$ J(u^*)= \min_{u\in W} J(u) $$
under the constraints $\left( E, S_s, S_i, L, M_{1,s}, M_{1,i}, \dots, M_{m,s}, M_{m,i} \right)$ is solution of \eqref{e2}.
\end{lem}

\begin{proof}
Let $u^n = \left(u^n_a, u^n_s, u^n_\kappa, u^n_1, \dots, u^n_{m_c}\right)$ be a minimizing sequence of controls in $W$, i.e.
$$
 \lim_{n\to+\infty} J(u^n) = \inf_{u\in W} J(u).
$$
This sequence is bounded, and sequential Banach–Alaoglu's theorem allows to extract a subsequence weak$-*$ convergent to $u^*=\left(u^*_a, u^*_s, u^*_\kappa, u^*_1, \dots, u^*_{m_c}\right)$ in $L^\infty([0, T]; W)$. 
For $n \in \N$, we denote  $U^n = (E^n, S_s^n, S_i^n, L^n, \linebreak M_{1,s}^n, M_{1,i}^n,\dots, M_{m,s}^n, M_{m,i}^n)$ the solution associated to the control $u^n$. 
Using Theorem \ref{thm1}, one can prove that $U^n$ is nonnegative and uniformly bounded. 
Then, $(U^n)'(t)$ is also bounded. Thus, from Arzelà–Ascoli theorem, we can extract a subsequence of $(U^n)_n$ that strongly converges to $U$ in 
$\mathcal{C}^0([0, T])$. Rewritting (\ref{e2}) in integral form, we get
\begin{eqnarray*}
  && (E(t)-E(0)) -  (E^n(t)-E^n(0)) \\
  && = \int_0^T \left(\sum_{j=1}^{m} \theta_j M_{j,i} - \omega_s S_s E - \mu_e E - u_a E\right) -  \left(\sum_{j=1}^{m} \theta_j M_{j,i}^n - \omega_s S_s^n E^n - \mu_e E^n - u_a^n E^n \right) dt 
\\
&& = \int_0^T \sum_{j=1}^{m} \theta_j (M_{j,i} - M_{j,i}^n) - \mu_e (E-E^n) - \omega_s (S_s (E-E^n) + (S_s - S_s^n) E^n) 
- u_a (E-E^n) + (u_a^n -u_a) E^n  dt
\end{eqnarray*}
that converges to $0$ as $n\to +\infty$. Dealing similarly for the remaining equations, the limit $U$ is then the solution of the system for the limit control $u^*$. 
Finally, by lower semi-continuity of $J$, we have
\begin{eqnarray*}
\liminf_{n\to +\infty} J(u^n) &=& \liminf_{n\to +\infty} \int_0^T  \sum_{j=1}^{m_c} M_{j,i}^n(t) + A_a (u_a^n)^2(t) + A_s (u_s^n)^2(t) + A_{\kappa} (u_{\kappa}^n)^2(t) +  \sum_{j=1}^{m_c} A_j (u_j^n)^2(t) dt
\\
&\geq& \int_0^T  \sum_{j=1}^{m_c} M_{j,i}(t) + A_a (u_a^*)^2(t) + A_s (u_s^*)^2(t) + A_{\kappa} (u_{\kappa}^*)^2(t) +  \sum_{j=1}^{m_c} A_j (u_j^*)^2(t) dt
\\
&=& J(u^*).
\end{eqnarray*}
\end{proof}

\noindent
Pontryagin’s maximum principle is used to find the best possible control for taking a dynamical system from one state to another. It states that it is necessary for any optimal control along with the optimal state trajectory to solve the so-called Hamiltonian system. We state the lemma below.
\begin{lem}
  There exist the adjoint variables $\Lambda = (\lambda_1,\lambda_2,\lambda_3,\lambda_4,\lambda_{1,s},\lambda_{1,i},\dots,\lambda_{m,s},\lambda_{m,i})$ of the system \eqref{e2} that satisfy the following backward in time system of ordinary differential equations:
  \begin{eqnarray*}
    \dfrac{d \lambda_1}{dt} &=& \Big(  \omega_s S_s + \mu_e + u_a \Big)\lambda_1(t)+\omega_s \beta_s S_s \lambda_2(t) -\omega_s \beta_s S_s \lambda_3 (t) \\
    \dfrac{d \lambda_2}{dt} &=& \omega_s  E \lambda_1(t) +\left( \omega_s \beta_s  E - {{\alpha}_s} \left( 1-\frac{2({S_s}+{S_i})}{{{\kappa}_s}/u_{\kappa}}\right) + \mu_s  + u_s\right)\lambda_2(t) -\omega_s \beta_s  E  \lambda_3 (t)\\
     \dfrac{d \lambda_3}{dt} &=& \alpha_s  \left(\frac{2({S_s}+{S_i})}{{{\kappa}_s}/u_{\kappa}}-1\right)\lambda_2(t) + \Big(\mu_s + \gamma_s + u_s \Big)\lambda_3 (t) -\nu_s \lambda_4(t)\\
     \dfrac{d \lambda_4}{dt} &=& \Big(  \sum_{j=1}^{m} \omega_j M_{j,s} + \mu_l + u_a \Big)\lambda_4(t)+ \omega_j \beta_j  M_{j,s}\lambda_{j,s}(t)- \omega_j \beta_j M_{j,s}\lambda_{j,i} (t) \\
    \dfrac{d \lambda_{j,s}}{dt} &=& \left(  \sum_{j=1}^{m} \omega_j  L  \right)\lambda_4(t)+ \Big(  \omega_j \beta_j L + \mu_j  \Big)\lambda_{j,s}(t)-  \omega_j \beta_j L \lambda_{j,i} (t)\\
    \dfrac{d \lambda_{j,i}}{dt} &=& -1-\sum_{j=1}^{m} \theta_j  \lambda_1(t) +\Big( (\mu_j + \gamma_j)  + u_j  \Big)\lambda_{j,i} (t)
\end{eqnarray*}
with the transversality condition $\Lambda(T)=0$.
\end{lem}
\begin{proof}
Using the Hamiltonian for \eqref{e2}, we have 
   \begin{align} \nonumber
  \mathcal{H} =& \sum_{j=1}^{m_c} M_{j,i}(t) + A_a u_a^2(t) + A_s u_s^2(t) + \sum_{j=1}^{m_c} A_j u_j^2(t)\\ \nonumber
   &+\lambda_1(t)\left( \sum_{j=1}^{m} \theta_j M_{j,i} - \omega_s S_s E - \mu_e E - u_a(t) E \right)\\ \label{e3}
    &+\lambda_2(t)\left( -\omega_s \beta_s S_s E + \alpha_s (S_s + S_i) \left( 1 - \frac{S_s + S_i}{\kappa_s/u_\kappa(t)}\right) - \mu_s S_s - u_s(t) S_s \right)\\\nonumber
   &+\lambda_3 (t)\Big( \omega_s \beta_s S_s E - (\mu_s + \gamma_s) S_i - u_s(t) S_i\Big)+\lambda_4(t)\left( \nu_s S_i - \sum_{j=1}^{m} \omega_j M_{j,s} L - \mu_l L - u_a(t) L \right)\\\nonumber
   &+ \sum_{j=1}^{m} \lambda_{j,s} (t) \Big(  -\omega_j \beta_j L M_{j,s} + \alpha_j - \mu_j M_{j,s} \Big)+ 
   \sum_{j=1}^{m} \lambda_{j,i} (t)\Big( \omega_j \beta_j L M_{j,s} - (\mu_j + \gamma_j) M_{j,i} - u_j(t) M_{j,i} \Big).
\end{align}
Therefore, finding the partial derivatives of $\mathcal{H}$ with respect to $E, S_s, S_i, L, M_{j,s}, M_{j,i}$, we have
\begin{align*}
   \dfrac{ \partial \mathcal{H}}{\partial E}=&-\Big(  \omega_s S_s + \mu_e + u_a(t) \Big)\lambda_1(t)-\omega_s \beta_s S_s \lambda_2(t) +\omega_s \beta_s S_s \lambda_3 (t) \\
   \dfrac{ \partial \mathcal{H}}{\partial S_s}=&- \omega_s  E \lambda_1(t) -\left( \omega_s \beta_s  E - {{\alpha}_s} \left( 1-\frac{2({S_s}+{S_i})}{{{\kappa}_s}}\right) + \mu_s  + u_s(t) \right)\lambda_2(t)-\Big(  u_s(t)-\omega_s \beta_s  E  \Big)\lambda_3 (t)\\
    \dfrac{ \partial \mathcal{H}}{\partial S_i}=&\alpha_s  \left(1-\frac{2({S_s}+{S_i})}{{{\kappa}_s}}\right)\lambda_2(t) - \Big(\mu_s + \gamma_s\Big)\lambda_3 (t) +\nu_s \lambda_4(t)\\
    \dfrac{ \partial \mathcal{H}}{\partial L}=& -\Big(  \sum_{j=1}^{m} \omega_j M_{j,s} + \mu_l + u_a(t) \Big)\lambda_4(t)- \omega_j \beta_j  M_{j,s}\lambda_{j,s}(t)+ \omega_j \beta_j M_{j,s}\lambda_{j,i} (t) \\
    \dfrac{ \partial \mathcal{H}}{\partial M_{j,s}}=& -\left(  \omega_j  L  \right)\lambda_4(t)- \Big(  \omega_j \beta_j L + \mu_j  \Big)\lambda_{j,s}(t)+  \omega_j \beta_j L \lambda_{j,i} (t)\\
    \dfrac{ \partial \mathcal{H}}{\partial M_{j,i}}=&1+ \theta_j  \lambda_1(t) -\Big( (\mu_j + \gamma_j)  + u_j(t)  \Big)\lambda_{j,i} (t)
\end{align*}
Then the adjoint system is defined by 
$$\dfrac{d \Lambda}{dt}=\left(- \dfrac{\partial \mathcal{H}}{\partial E}, - \dfrac{\partial \mathcal{H}}{\partial S_s}, - \dfrac{\partial \mathcal{H}}{\partial S_i}, - \dfrac{\partial \mathcal{H}}{\partial L}, - \dfrac{\partial \mathcal{H}}{\partial M_{1,s}}, - \dfrac{\partial \mathcal{H}}{\partial M_{1,i}}, \hdots, - \dfrac{\partial \mathcal{H}}{\partial M_{m,s}}, - \dfrac{\partial \mathcal{H}}{\partial M_{m,i}} \right)^T $$ 
\end{proof}

\begin{thm}
  The optimal control variables are given, for $1\leq j \leq m_c$, by
  \begin{align*}
      u^*_a(t)=&\max \left( 0, \min \left( \dfrac{\lambda_1(t) E(t) + \lambda_4(t)L(t)}{2 A_a}, u_{Amax} \right) \right)
      \\
    u^*_s(t)=&\max \left( 0, \min \left( \dfrac{\Big(\lambda_2(t) S_s(t)  +\lambda_3(t) S_i(t)}{2 A_s}, u_{Smax} \right) \right)
    \\
        u^*_{\kappa}(t)=&\max \left( 0, \min \left( \dfrac{\alpha_s\lambda_2(t) (S_s +S_i)}{2\kappa_s A_{\kappa}}, u_{Kmax} \right) \right)
    \\
    u^*_j(t)=&\max \left( 0, \min \left( \dfrac{\lambda_{j,i}(t) M_{j,i}(t)}{2 A_j}, u_{Mmax,j} \right) \right).
  \end{align*}
\end{thm}
\begin{proof}
By the Pontryagin maximum principle, the optimal control $u^*$ minimizes, at each instant $t$, the Hamiltonian given by \eqref{e3}. We have

$$ \dfrac{\partial \mathcal{H}}{\partial u_k}=0, \quad \text{for all } \quad k=a, s, \kappa, j \quad \text{at} \quad u_k=u_k^* .$$
Therefore, we get 
\begin{align*}
    \dfrac{\partial \mathcal{H}}{\partial u_a}=&2A_a u_a (t)-\lambda_1(t)E(t) -\lambda_4(t)L(t),
   \qquad
     \dfrac{\partial \mathcal{H}}{\partial u_s}=2A_s u_s(t)-\lambda_2(t)S_s(t) - \lambda_3(t)S_i(t)
     \\
       \dfrac{\partial \mathcal{H}}{\partial u_j}= & 2 A_j u_j(t)-\lambda_{j,i}(t)M_{j,i}(t),
    \qquad
         \dfrac{\partial \mathcal{H}}{\partial u_{\kappa}}=  2 A_{\kappa} u_{\kappa}(t)-\lambda_{2}(t)\dfrac{\alpha_s}{\kappa_s}(S_s+S_i)^2.
\end{align*}
\end{proof}



\section{Numerical results}

\subsection{Comparing controls: treatment versus molluscicide versus mechanical control}


We consider two mammalian populations: humans and bovines (including cattle and carabao). Optimal control strategies are initiated upon reaching endemic equilibrium. This involves determining the most effective chemotherapy treatment (Mass Drug Administration with Praziquantel) for humans and bovines, mechanized farming to prevent infected bovines from spreading schistosome-contaminated feces in ricefields, implementing optimal molluscicide applications to target the snail population, and employing mechanical control methods such as habitat modification (e.g., concreting irrigation ditches) to reduce snail numbers. These controls are also implemented in the backdrop of uncontrolled rodents populations. 

    \begin{figure}
    \centering
    \includegraphics[scale = 0.8]{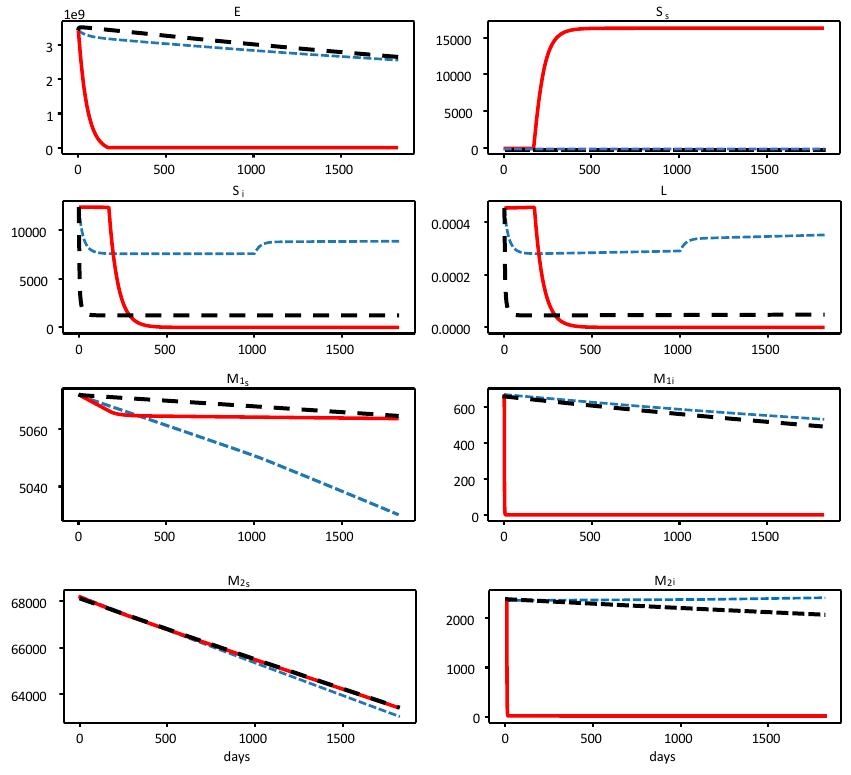}
    \caption{Solution with respect to time under optimal control. The red curves represent the solution with optimal treatment on human $u_1^*$ and bovines $u_2^*$ only. The black curves represent the solution with optimal mechanical control $u_\kappa^*$.
    The blue curves represent the solution with optimal aquatic control $u_s^*, u_a^*$.}
    \label{fig2}
\end{figure}

\begin{figure}
    \centering
    \includegraphics[scale = 0.5]{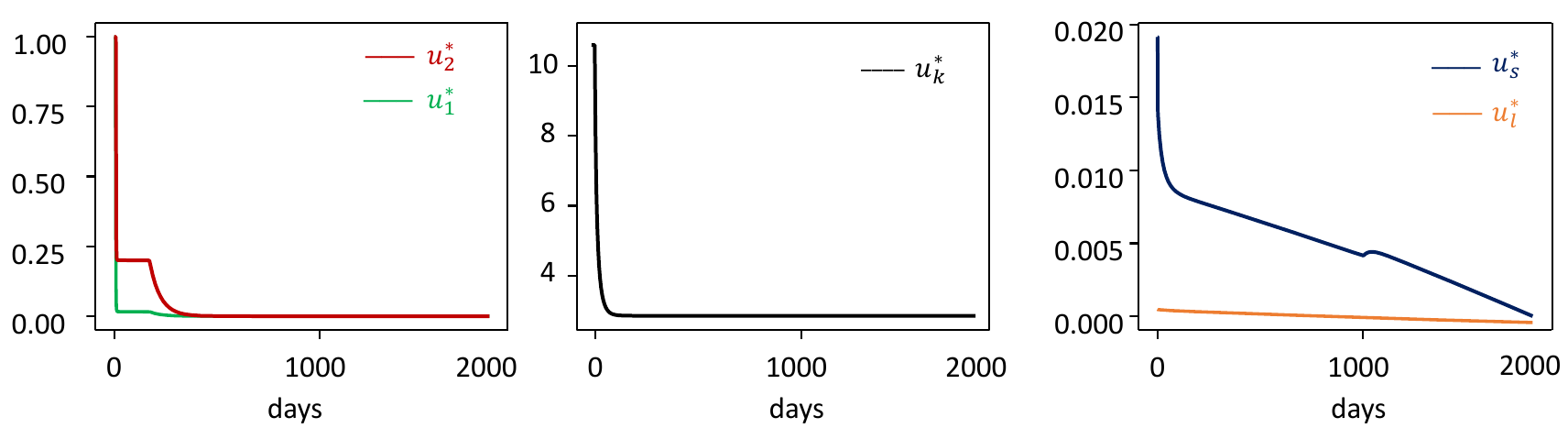}
    \caption{The red and green curves represent the optimal treatment on humans $u_1^*$ and bovines $u_2^*$ only. The black curve represents the optimal mechanical control $u_\kappa^*$.
    The blue and orange curves represent the optimal aquatic control $u_s^*, u_a^*$.}
    \label{fig3}
\end{figure}

In Figure \ref{fig2}, the red curve illustrates the solution when control measures are exclusively applied to the treatment of humans ($u_1^*$) and bovines ($u_2^*$). The treatment represents the most effective approach, resulting in the steepest and fastest decline in the number of infected humans and bovines. Subsequently, there is a decrease in the number of schistosome eggs in the definitive host's feces, followed by a reduction in the number of infected snails and cercariae. Susceptible humans and bovines gradually converge toward the Disease-Free Equilibrium (DFE), indicating the effectiveness of the Mass Drug Administration (MDA). However, it's crucial to note that current MDA programs in the Philippines typically target humans only, neglecting bovines. Therefore, the full potential of MDA remains underappreciated and should ideally include bovines in its treatment regime. In addition, we notice that the optimal control for humans ($u_1^*$) and infected bovines ($u_2^*$) exhibit a hump during the first few months of application and gradually decrease until treatment cessation (see Figure \ref{fig3} ). This finding suggests the necessity for more robust and prolonged control measures over bovines compared to humans. A strategy never yet explored by the health authorities in the Philippines.  

On the other hand, it's crucial to acknowledge that depending solely on aquatic control through molluscicide doesn't ensure the eradication of the disease and shows subpar performance in our simulations (blue curve in Figure \ref{fig2}). Furthermore,  the control is less effective in the sense that the number of infected humans and bovines decreases slowly, and control must be continuously applied, as observed in Figure \ref{fig3} on the right. This point warrants attention from policymakers and local government units to discourage reliance on this method. Molluscicide application not only poses environmental risks (such as toxicity to other flora and fauna and rapid reestablishment of snail populations) but also imposes significant economic burdens. 

Finally, the black curve in Figure \ref{fig2} illustrates the outcome of implementing mechanical control measures like modifying snail habitat. While this approach effectively reduces the population of snails and larvae (miracidia and cercariae), it demonstrates slower progress in decreasing the number of infected humans and bovines. As depicted in Figure \ref{fig3}, the optimal mechanical control strategy ($u_\kappa^*$) should initially remain constant and then gradually decrease but cannot be ceased. The carrying capacity of the snails is divided by $3.3$, meaning that the snails' habitat is reduced by $70\%$. Unlike the MDA approach, it underscores that habitat modification primarily impacts snail and cercariae populations rather than directly curtailing human and bovine infections. This highlights its limitations. Therefore, although habitat modification is a viable alternative and has succeeded in countries like Japan, it should not be relied solely on to eradicate schistosomiasis. Additionally, it's essential to acknowledge its downsides, which, aside from being costly, have also been outlined in the paper by \cite{bay2022total}. 

\subsection{Compliance with control upper bounds}

Enhancing compliance among controls is paramount for achieving successful and sustainable control of schistosomiasis in Lake Mainit \citep{inobaya_mass_2018}. This ensures that interventions are not only effective and targeted but also resilient against potential resurgence. Therefore, we test the compliance by varying the maximum bound of the optimal control.

\begin{figure}[h]
    \centering
    \includegraphics[scale = 0.5]{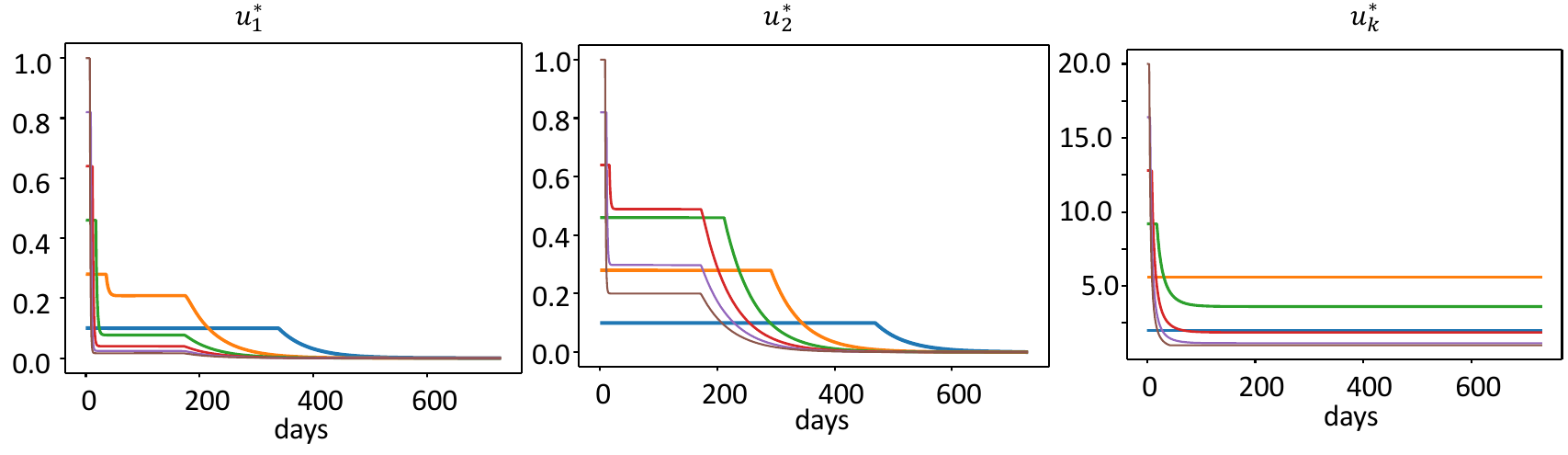}
    \caption{Left: Curves depict the optimal treatment for human ($u_1^*$). Center: Curves illustrate the optimal treatment for bovines ($u_2^*$). Right: Curve represents the optimal mechanical control ($u_\kappa^*$) with control compliance varying from $10\%$ to $100\%$.}
    \label{fig4}
\end{figure}

In Figure \ref{fig4}, it is evident that the duration of control application increases as compliance decreases. Specifically, in the case of humans, optimal efficiency and cost-effectiveness of treatment require a compliance rate of at least 80\%. Failure to achieve this level of compliance necessitates continuous control application for at least 200 days. For bovines, regardless of compliance percentage, treatment should be maintained for a minimum of 200 days. Notably, the lower the compliance, the longer the duration of control application. Additionally, it is important to highlight that mechanical control should always be implemented below a certain threshold. This is due to the permanent reduction in the snails' habitat, which cannot be restored to its original capacity.


\subsection{Consequences of uncontrolled mammal population}

We consider three populations of mammals, namely humans, bovines, and rodents population. Optimal control is applied to humans and bovines, but not the rodents.

\begin{figure}[h]
    \centering
    \includegraphics[scale = 0.5]{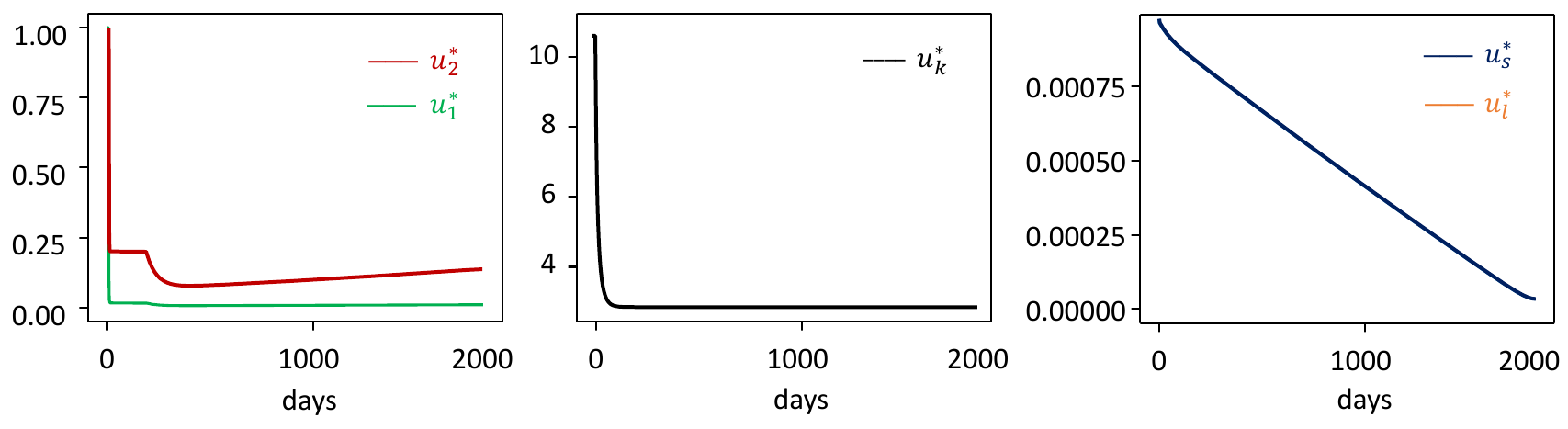}
    \caption{The red and green curves represent the optimal treatment on humans $u_1^*$ and bovines $u_2^*$ only. The black curve represents the optimal mechanical control $u_\kappa^*$.
    The blue and orange curves represent the optimal aquatic control $u_s^*, u_a^*$.}
    \label{fig5}
\end{figure}

When an uncontrolled rodent population is added, reducing the number of infected humans and bovines becomes more challenging. As observed in Figure \ref{fig5}, after a phase of decreasing treatments to bovines, it becomes necessary to gradually increase their treatment rate again. Similarly, the mechanical control measures must be sustained.

\begin{figure}[h]
    \centering
    \includegraphics[scale = 0.5]{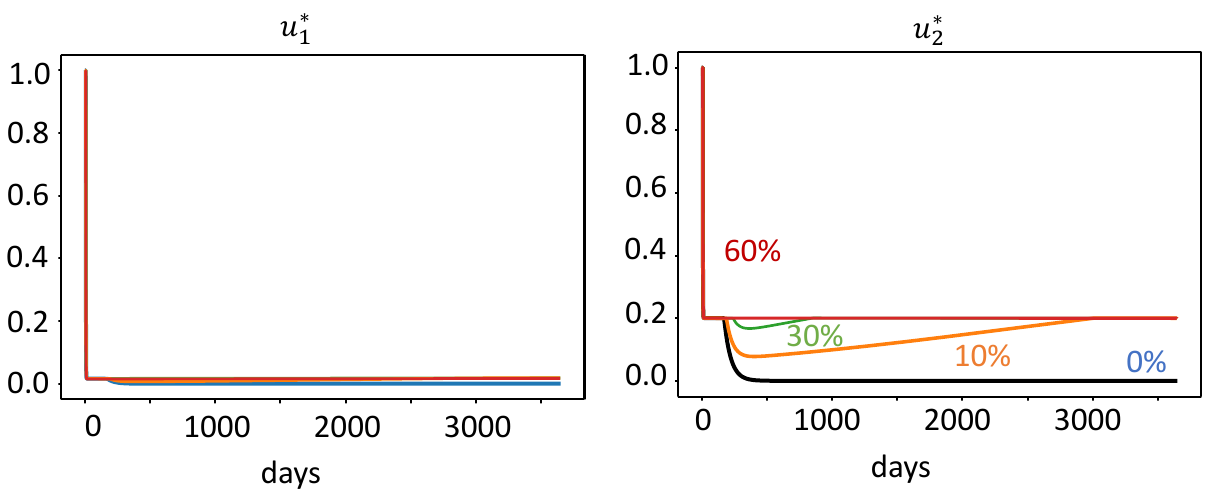}
    \caption{Left, the curves represent the optimal treatment on human $u_1^*$. Right, the curves represent the optimal treatment on bovine $u_2^*$ with respect to the prevalence on rodents varying between $0\%$ and $60\%$.}
    \label{fig6}
\end{figure}

According to Figure \ref{fig6}, the larger the number of infected and uncontrolled rodents, the more important it is to extend the control of bovines. As long as rodents are present, it becomes imperative to maintain treatment for at least $20\%$ of the bovine population.

\subsection{Integrated control strategies: mechanical, aquatic and treatment}
This section investigates an integrated approach that allows us to assess the combined impact of our strategies to control schistosomiasis transmission.
\begin{figure}[h]
    \centering
    \includegraphics[scale = 0.7]{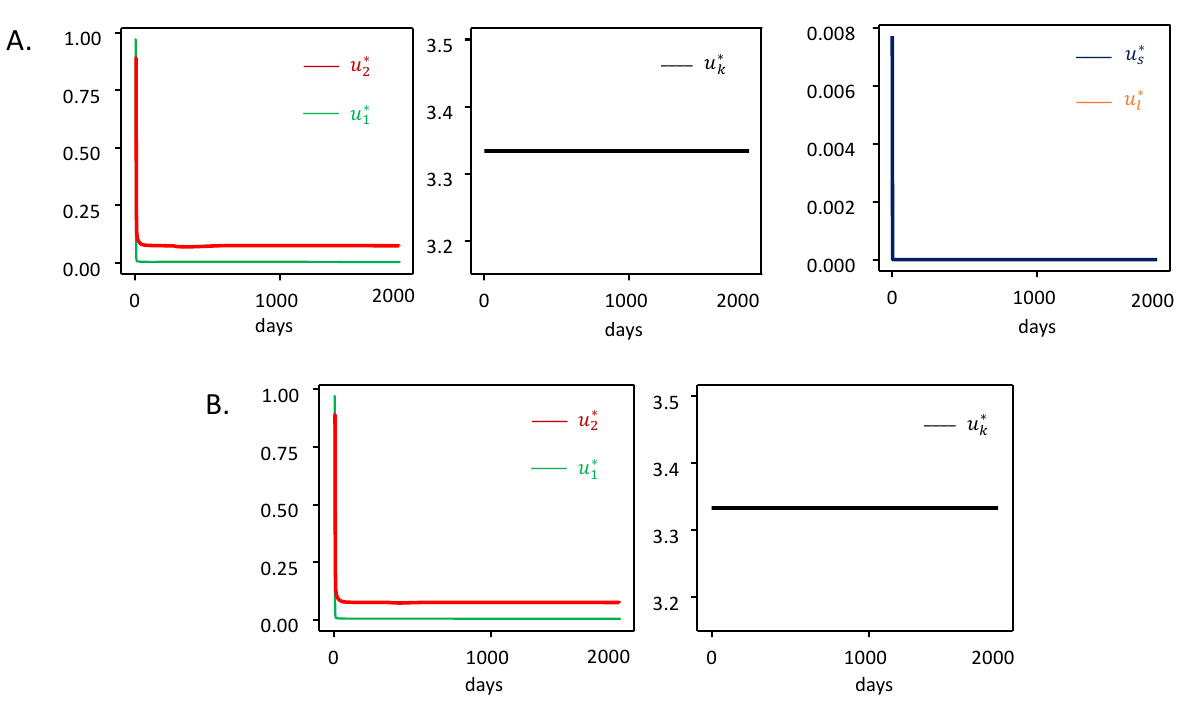}
    \caption{The red and green curves represent the optimal treatment on humans $u_1^*$ and bovines $u_2^*$ only. The black curve represents the optimal mechanical control $u_\kappa^*$.
    The blue and orange curves represent the optimal aquatic control $u_s^*, u_a^*$. A. All controls are implemented simultaneously.
B. MDA and mechanical controls are concurrently applied.}
    \label{fig7}
\end{figure}

Figure \ref{fig7} demonstrates the effectiveness of integrating multiple control measures in minimizing infection spread among humans and bovines. Our simulation assumes the use of a molluscicide with a $69\%$ efficacy rate \citep{Jia2019}, human treatment with an $89\%$ efficacy rate \citep{Hailegebriel21}, bovine treatment with a $97\%$ efficacy rate \citep{Wang06}, and a $70\%$ reduction in snail carrying capacity.

Our observation indicates that while molluscicide application and human treatment can be discontinued relatively quickly, ongoing bovine treatment is necessary, and the reduction in snail-carrying capacity remains consistently effective. Thus, to eliminate the infection, all preventative measures are necessary.

\section{Discussion}

The complexity of schistosomiasis control is intensified by the diverse range of potential hosts, which includes domestic animals such as water buffalo, cattle, sheep, dogs, cats, horses, pigs, and rodents \citep{van2017rodents}. These animals significantly influence public health due to their role as reservoir hosts for \textit{Schistosoma} eggs. Control efforts become crucial in ecosystems like the lake-ricefield interfaces, where human and bovine populations are dominant. For instance, water buffaloes and cattle can excrete up to 60 kg of stool per individual per day \citep{gray2009cluster}, a stark contrast to the 325 g that humans typically excrete \citep{wang2005transmission}. This discrepancy results in these animals becoming larger reservoirs for \textit{Schistosoma} eggs, thereby complicating control efforts.

In areas such as Lake Mainit, the \textit{Schistosoma} bovine contamination index (BCI) can soar as high as 104,750 per bovine \citep{estano2023prevailing}, demonstrating the severe impact of these infections. 

Additionally, Praziquantel (PZQ) remains the treatment of choice for various mammalian hosts, known for its cost-effectiveness at 40 mg/kg per single dose \citep{who2013schistosomiasis}. Although effective, PZQ does not kill immature worms during treatment, necessitating subsequent doses every two to four weeks to maximize its effectiveness \citep{woolhouse1991application}. While PZQ treatment for bovines is implemented in some countries, it has not been widely adopted in the Philippines, particularly in areas around Lake Mainit. Efforts to eliminate schistosomiasis in domestic animals, such as those observed in the mountainous regions of China through monthly PZQ treatments, provide a viable model for consideration \citep{li2015elimination}. Furthermore, the Department of Agriculture (DA) and the Department of Health (DOH) in the Philippines are developing strategies to combat schistosomiasis in animals, including the creation of a bovine vaccine \citep{leonardo2016schistosomiasis}. Recent studies, such as those involving the SjCTPI DNA vaccine administered through intramuscular injections in bovines, have shown promising results, leading to a significant reduction in human schistosomiasis cases across 18 villages in Northern Samar, affecting 18,221 residents \citep{ross2023first}. Despite these advancements, the nationwide implementation of bovine vaccination as an additional treatment method alongside Praziquantel (PZQ) has not yet been established in endemic areas.

Given these considerations, we have developed a nonlinear dynamic model incorporating a range of strategies, such as chemotherapy, molluscicides, habitat modification of snails, and mechanical interventions. This model is designed to ascertain the most effective control or combination of controls. 

Among the control strategies examined, chemotherapy for both humans and bovines emerged as the most effective, rapidly decreasing infection rates and moving populations towards a disease-free equilibrium. However, reliance on chemotherapy alone is insufficient due to the continuous exposure to infected water sources. Mechanical controls, such as habitat modification, and aquatic controls, such as molluscicides, showed limited effectiveness when implemented in isolation. Integrated control strategies, combining chemotherapy, mechanical, and aquatic interventions, proved to be the most effective approach.

While treating humans is crucial, our simulations indicate that including bovines in the treatment regimen significantly enhances the effectiveness of control measures. Mechanical control measures effectively reduce snail populations and cercariae, but their impact on infected human and bovine populations is less immediate. Aquatic controls require sustained application and are not cost-effective in isolation. Moreover, the presence of uncontrolled rodent populations further complicates control efforts, necessitating more intensive and prolonged control measures.

Our numerical results indicate that a multi-pronged approach, including enhanced compliance and consideration of uncontrolled mammal populations, such as rodents, is crucial for a successful disease management strategy.

\subsection{The lake-ricefield route of the life cycle of Schistosoma japonicum in Lake Mainit: Basis for intervention}

Rice fields serve as critical transmission points for schistosomiasis in humans, with the risk intensifying when these fields are cohabited with infected snails and unprotected farmers. The association between schistosomiasis risk and unshielded exposure to contaminated waters, particularly where bovines with the infection graze, has been demonstrated in the rice fields of Lake Mainit. Beyond the fields, grazing and resting spots for livestock also contribute significantly to the spread of the disease.

Transmission dynamics in Lake Mainit are largely driven by human contact with cercaria-infested waters during routine agricultural and domestic activities, including fishing, washing, and bathing. The area's appeal as a recreational spot, due to its touristic potential, further escalates the risk of exposure. Seasonal floods exacerbate this situation by dispersing vectors and parasites over a larger area, a phenomenon that is not unique to Lake Mainit but is common across similar ecosystems, thereby increasing the vulnerability to schistosomiasis.

Addressing the spread of schistosomiasis necessitates targeted interventions aimed at bovine reservoirs, known for their role in the fecal dissemination of schistosome eggs. Effective strategies may include:
\begin{enumerate}
    \item Providing infected bovines with antischistosomal and deworming medications. Given that a safe dosage of PZQ or other deworming drugs is eventually developed and adapted for treating infected bovines, reducing the parasite burden in these animals is possible. This, in turn, can lower the risk of environmental contamination with \textit{Schistosoma} eggs, thus interrupting the transmission cycle.
    \item Advancing agricultural mechanization to limit bovine contact with infected waters, potentially replacing buffaloes with machinery. Advancing agricultural mechanization to limit bovine contact with infected waters, potentially replacing buffaloes with machinery, is a proactive approach that reduces the risk of disease transmission and enhances agricultural productivity and yield. However, implementing this strategy entails substantial costs and necessitates strong financial backing and political commitment from relevant agencies for successful execution.
    \item Restricting domestic animal movement, especially bovines, by erecting barriers and establishing water-free zones to prevent schistosomes from feces from hatching, miracidia from swimming towards the snail vector (if present in the area), or the cercaria from penetrating the skin of the mammalian host.
    \item Promoting hygiene and encouraging the use of protective gear for individuals in high-risk environments such as rice fields. Promoting proper sanitation practices, such as building and using latrines to prevent fecal contamination of water sources, is crucial in reducing the environmental reservoir of Schistosoma parasites.
    \item Health Education and Awareness. Enhancing awareness that cows, carabaos, and other animals (including rodents, dogs, cats, pigs, horses, and goats) can be reservoir hosts for schistosomiasis. 

\end{enumerate}
When these measures are effectively executed, they have the potential to substantially decrease schistosomiasis cases within affected communities.

\section*{Funding}
This study was financially supported by the Department of Science and Technology-Philippine Council for Health Research and Development (\href{https://www.pchrd.dost.gov.ph/}{DOST-PCHRD}) and the Department of Health Caraga, which funded two years of data gathering. In the final years, the Philippine authors received funding from the DOST-Office of the Assistant Secretary for International Cooperation (OASEC-IC) under the Philippine-French Hubert Curien Partnership (PHC) “Science for the People” (SFTP) program for a study visit to France. Additionally, the Laotian author was funded by the International Science Programme (ISP) at Uppsala University, Sweden, for the project activity "South-East Asia Mathematical Network (SEAMaN - NUOL)", supporting full PhD preparation at the Université de Picardie Jules Verne. Further, CIMPA provided funding for attending a CIMPA School at Caraga State University, which facilitated the Laotian author's visit to Lake Mainit, Philippines, for validation and appreciation. The funders of the study had no role in the study design, data collection and analysis, decision to publish, or preparation of the manuscript.

\section*{CRediT authorship contribution statement }

\begin{table}[H]
\scriptsize
\centering
\begin{tabular}{>{\bfseries}p{4cm}p{10cm}}
\toprule
\textbf{Contribution} & \textbf{Authors} \\
\midrule
Conceptualization & Youcef Mammeri, Joycelyn Jumawan, Jess Jumawan, Jayrold Arcede \\
Methodology & \textbf{Mathematics:} Youcef Mammeri, Jayrold Arcede, Boausy Doungsavanh; \textbf{Biology:} Joycelyn Jumawan, Jess Jumawan, Leonard Estaño \\
Software & Youcef Mammeri \\
Validation & Youcef Mammeri, Joycelyn Jumawan,  Leonard Estaño \\
Formal Analysis & Youcef Mammeri, Jayrold Arcede, Boausy Doungsavanh, Joycelyn Jumawan, Jess Jumawan, Leonard Estaño \\
Investigation & Joycelyn Jumawan, Jess Jumawan, Leonard Estaño \\
Resources & Youcef Mammeri, Jayrold Arcede, Joycelyn Jumawan \\
Data Curation & Joycelyn Jumawan, Jess Jumawan, Leonard Estaño \\
Writing - Original Draft & Boausy Doungsavanh, Youcef Mammeri, Jayrold Arcede \\
Writing - Review, Editing \& Visualization & Youcef Mammeri, Jayrold Arcede, Boausy Doungsavanh, Joycelyn Jumawan, Jess Jumawan, Leonard Estaño \\
Supervision & \textbf{Mathematics:} Youcef Mammeri, Jayrold Arcede; \textbf{Biology:} Joycelyn Jumawan, Jess Jumawan \\
Project Administration & Jayrold Arcede, Joycelyn Jumawan \\
Funding Acquisition & Youcef Mammeri, Jayrold Arcede, Joycelyn Jumawan \\
\bottomrule
\end{tabular}
\end{table}

\section*{Declaration of Competing Interest}
The authors declare that they have no known competing financial interests or personal relationships that could have appeared to influence the work reported in this paper.

\section*{Acknowledgments}

We would like to express our gratitude to all the people from the Lake Mainit community who collaborated in the development of this study. We also wish to thank the DOST-Caraga Health Research Consortium for their assistance with funding and the Department of Health Caraga for logistics during data gathering.



\appendix
\section{Parameters estimation}

This appendix outlines the parameters and the methodology employed for their calibration. The dataset utilized was sourced from the Department of Health of the Philippines, with additional data collected by the Department of Biology at Caraga State University \citep{birth_stats,carabao_stats}. Specifically, the dataset covers the prevalence of schistosomiasis in both humans and bovines across the primary barangays surrounding Lake Mainit.
\\
We use a quasi-Newtonian method and an approximation of a Bayesian computation \citep{csillery2010} to estimate ten parameters. 
Figure \ref{figcalib} depicts the fitted prevalence of infected humans and infected carabaos for the posterior distribution. 
The list of the parameters is given in Table \ref{tabparam}.

\begin{figure}[h] 
  \centering
  \includegraphics[scale = 0.7]{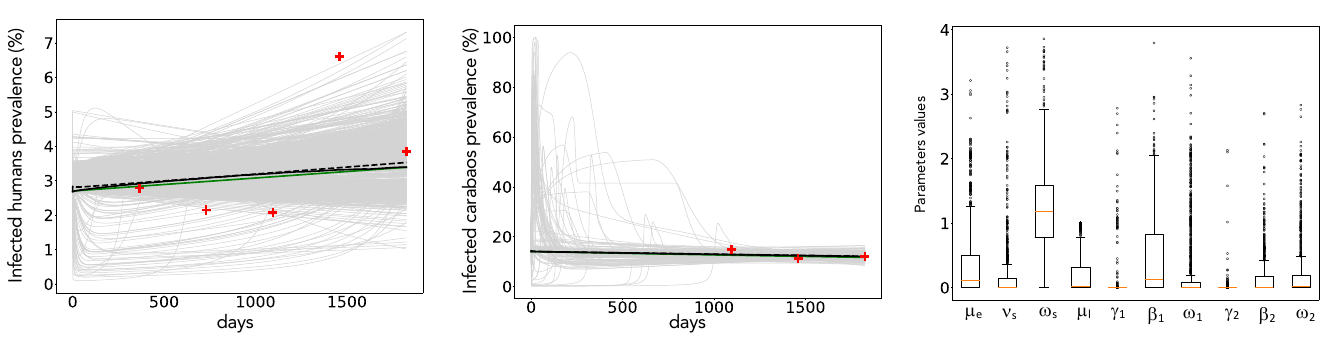}
  \caption{Fitted prevalence of infected humans and infected carabaos (in grey) for the posterior distribution. The straight line indicates the median, and dotted line indicates the mean.
   The straight green line indicates the parameters chosen for the optimal control tests.}
  \label{figcalib}
\end{figure}

\begin{landscape}
\begin{table}[h] 
  \label{tabparam}
  \centering
\begin{tabular}{|c|c|c|c| }
\hline
{\bf Parameters  } & {\bf Description } & {\bf Mean value (variance) } & {\bf References}\\
\hline \hline
$\alpha_1$ & birth rate of carabao per day & $8.235 \times 10^{-8}$  & \cite{carabao_stats} \\
\hline
$\mu_1$ & natural death rate of carabao per day  & $1.218\times 10^{-7}$ & \cite{carabao_stats}  \\
\hline
$\theta_1$ & fecal oviposition from carabao per day  & $104750$  & \cite{jumawan_prevalence_2021}  \\
\hline
$\gamma_1$ & death rate of carabao due to schistosomiasis per day  &  $0.00486$ ($0.00680$) &  fitted \\
\hline
$\beta_1$   & probability infection of carabao & $0.47124$ ($0.37605$) & fitted  \\
\hline
$\omega_1$& contact rate between carabao and cercariae per day& $0.24692$ ($0.30996$) & fitted \\
\hline
\hline
$\alpha_2$ & birth rate of human per day & $0.00278$  & \cite{birth_stats}\\
\hline
$\mu_2$ & natural death rate of human per day  & $3.914 \times 10^{-5}$  & \cite{birth_stats} \\
\hline
$\theta_2$ & fecal oviposition from human per day  & $250$ & \cite{feng_estimation_2004}  \\
\hline
$\gamma_2$ & death rate of human due to schistosomiasis per day  & $0.00898$ ($0.01093$) & fitted \\
\hline
$\beta_2$   & probability infection of human & $0.01591$ ($0.09510$) & fitted \\
\hline
$\omega_2$& contact rate between human and cercariae per day& $0.19972$ ($0.15995$) & fitted  \\
\hline
\hline
$\alpha_s$ & birth rate of snail per day & $0.06886$ & \cite{bogitsh_chapter_2019} \\
\hline
$\mu_s$ & natural death rate of snail per day  & $0.0035$ & \cite{bogitsh_chapter_2019} \\
\hline
$\kappa_s$& carrying capacity of snail per $m^2$ & $17300$ & $10\%$ of the lake Mainit surface \\
\hline
$\gamma_s$ & death rate of snail due to schistosomiasis per day  & $0.016$  & \cite{ishikawa_modeling_2006}\\
\hline
$\beta_s$   & probability infection of snail &  $10^{-4}$ & \cite{feng_estimation_2004}\\
\hline
$\omega_s$& contact rate between snail and cercariae  per day& $1.17617$ ($0.56018$) & fitted \\
\hline
$\nu_s$ & transformation rate of larvae inside snails  per day  & $0.20865$ ($0.25194$) &  fitted\\
\hline
\hline
$\mu_e$ & natural death rate of egg & $0.38527$ ($0.33031$) & fitted\\
\hline
$\mu_l$ & natural death rate of cercariae& $0.23407$ ($0.13348$) & fitted \\
\hline
\end{tabular}
\end{table}
\end{landscape}

\end{document}